\documentclass[a4paper,10pt]{amsart}

\usepackage{amsfonts,latexsym,rawfonts,amsmath,amssymb,amsthm, mathrsfs, lscape}
\usepackage{graphicx}
\usepackage[all]{xy}
\usepackage[english]{babel}
\usepackage[utf8]{inputenc}
\usepackage{xfrac}

\usepackage{cite}

\usepackage{array, tabularx}

\usepackage{bm}
\bmdefine{\bX}{X}
\bmdefine{\ba}{$\alpha$}

\usepackage{setspace}
\usepackage{textcomp}

\usepackage{color}

\usepackage{xparse}
\usepackage{soul}

\usepackage{graphicx}

\newcommand{\referenza}{}

\newtheorem{thm}{Theorem}[section]
\newtheorem*{thm*}{Theorem }
\newtheorem*{thm**}{Theorem}

\newtheorem*{cor*}{Corollary \referenza}
\newtheorem{lem}[thm]{Lemma}
\newtheorem*{lem*}{Lemma \referenza}

\newtheorem*{clm*}{Claim \referenza}
\newtheorem{prop}[thm]{Proposition}
\newtheorem*{prop*}{Proposition \referenza}

\newtheorem*{ass*}{Assumption \referenza}

\newtheorem*{conj*}{Conjecture \referenza}

\newtheorem{rmk}[thm]{Remark}

\theoremstyle{plain}

\def \R {\mathbb R}

\allowdisplaybreaks

\usepackage[hypertexnames=false,
backref=page,
    pdfpagemode=UseNone,
    breaklinks=true,
    extension=pdf,
    colorlinks=true,
    linkcolor=blue,
    citecolor=blue,
    urlcolor=blue,
]{hyperref}

\makeatletter
\@namedef{subjclassname@2020}{
  \textup{2020} Mathematics Subject Classification}
\makeatother

\begin{document}

\title{Solvability of a class of fully nonlinear elliptic equations on tori.}

\author{Elia Fusi}
\address{	Dipartimento di Matematica "Giuseppe Peano"\\
	Universit\`a di Torino,
	Via Carlo Alberto 10,
	10123 Torino, Italy}
\email{elia.fusi@unito.it}


\begin{abstract}
We study the solvability of a class of fully nonlinear equations on the flat torus. The equations arise in the study of some Calabi-Yau type problems in torus bundles.  
\end{abstract}

\maketitle

\section{Introduction}

After Yau proved the Calabi conjecture in \cite{17}, some new Calabi-Yau type  equations were introduced on non-K\"ahler manifolds. Donaldson in \cite{4} formulated a project on compact 4-dimensional symplectic manifolds which is based on a Calabi-Yau equation  on compact almost-K\"ahler  manifolds. The problem was studied by Weinkove in \cite{16} and by Tosatti, Weinkove and Yau in \cite{12} assuming extra hypothesis on the curvature and on the torsion of the almost-K\"ahler metric. Later, Tosatti and Weinkove in \cite{13} solved the almost-K\"ahler Calabi-Yau equation on the Kodaira-Thurston manifold assuming the data to be invariant under the action of a 2-dimensional torus. Buzano, Fino and Vezzoni in \cite{2} generalized the Tosatti-Weinkove theorem to $S^1$-invariant data. 
In the latter case,  the problem reduces to the study of the following equation on a $3$-dimensional torus $T^3$
\begin{equation}\label{eq:1}
(1+u_{xx})(1+u_{yy}+u_{tt}+u_t)-u_{xy}^2-u_{xt}^2=e^{f}\,,
\end{equation} 
where $f\in C^{\infty}(T^3)$ is given and satisfies $$
\int_{T^3}e^fdV=1\,.$$ This last result was then extended by Tosatti and Weinkove in \cite{14} considering different almost-K\"ahler structures.

In \cite{1} Alesker and Verbitsky introduced a Calabi-Yau problem in HKT geometry and formulated the so called
{\em quaternionic Calabi conjecture}. Gentili and Vezzoni in \cite{6} confirmed the conjecture on 8-dimensional nilmanifolds endowed with an Abelian HKT structure and assuming the data invariant by the action of  a $3$-dimensional torus. 
Under these assumptions, the problem reduces to the following 
equation on a $5$-dimensional torus $T^5$ \begin{equation}\label{eq:2}
(1+u_{55})\left(1+\sum_{i=1}^4u_{ii}\right)-\sum_{i=1}^4u_{i5}^2=e^f\,,
\end{equation} where again $f\in C^{\infty}(T^5)$ is given and satisfies $$
\int_{T^5}e^fdV=1\,.
$$ 

The goal of the present paper is to study a class of PDEs on the $n$-dimensional torus including \eqref{eq:1} and \eqref{eq:2} as special cases. Namely, we consider the following type of equations on the $n$-dimensional torus $T^n$
\begin{equation}\label{eq:3}
(1+u_{nn}+G(\nabla u))\left(1+\sum_{i=1}^{n-1}u_{ii}+F(\nabla u)\right)-\sum_{i=1}^{n-1}u_{in}^2=e^{f}\,,
\end{equation}
where $n>2$,  $f\in C^{\infty}(T^n)$ and $F$, $G$ are smooth functions of the gradient of $u$. In the case in which $n=2$, many equations, even more general than (\ref{eq:3}), were studied, see, for instance, \cite{3},\cite{5},\cite{8} and \cite{15}. 

\medskip
Our main result is the following:
\begin{thm}\label{thm:1.1} 
Let $X,Y$ be two smooth vector fields with constant coefficients on $T^n$, $n>2$, such that, $\forall v\in C^{\infty}(T^n)$, $F(\nabla v)=X^iv_i$ and $G(\nabla v)=Y^iv_i.$ 
%Moreover, suppose that \begin{enumerate}\item\label{H:1} $Y$ is constant and X does not depend on $x_n$, \item\label{H:2} the matrix $\frac{\partial X}{\partial x}=(\frac{\partial X^i}{\partial x_j})_{i,j}$ is negative semidefinite,\item\label{H:3}  $0=\sum_{i=1}^{n-1}Y^i\frac{\partial X^j}{\partial x_i},$ $\forall j=1,\ldots, n-1$. \end{enumerate}
Then, equation (\ref{eq:3}) has a unique solution $u\in C^{\infty}(T^n)$ such that $$
\int_MudV=0\,.
$$
 \end{thm} 

The proof of Theorem \ref{thm:1.1} is based on the continuity method and it will be 
obtained as follows: in Section 2, we prove some preliminary results about solutions of (\ref{eq:3}); in Section 3, we prove the $C^0$ estimate using the Aleksandrov-Bakelman-Pucci maximum principle; in Section  4 we prove the $C^0$ estimate for the Laplacian of the solutions and higher order estimates by modifying an argument from \cite{6} and in \cite{2}; in Section 5, we conclude the proof of Theorem \ref{thm:1.1}. 

In the last section, more general equations are taken into account.

{\itshape Acknowledgements.} The author is very grateful to Professor Luigi Vezzoni for his supervision and support during the work. Many thanks are also due to Giovanni Gentili and Federico Giusti for stimulating discussions and  suggestions. 

The author is supported by GNSAGA of INdAM. 
\section{Preliminaries.}
In the following, we will always denote the $n$-dimensional torus as $M$, assuming  $n> 2$, and make use of the Einstein convention over repeated indexes. Moreover,  we will identify functions on $M$ with functions on $\mathbb R^n$  which are $1$-periodic in each variable and denote with $\{x_1,\ldots, x_n\}$ the standard coordinates on $\mathbb R^n$, unless otherwise stated.

In this section, $F$ and $G$ will only be  smooth functions of the gradient of $u$ such that $F(0)=G(0)=0$. For the sake of simplicity, given $u\in C^{2}(M)$, we introduce the following notation:
 $$
A=1+u_{nn}+G(\nabla u)\,,\quad B=1+\sum_{i=1}^{n-1}u_{ii}+F(\nabla u)\,.
$$
 So, fixed  $f\in C^{\infty}(M),$ equation (\ref{eq:3}) can be written in a more compact way, which is:
 \begin{equation}\label{eq:equazione}
 AB-\sum_{i=1}^{n-1}u_{in}^2=e^{f}\,,
 \end{equation} where $u\in C^{\infty}(M)$ is the unknown. We will search for solutions $u$ such that $$
 \int_MudV=0\, ,
 $$ where $dV$ denotes the standard volume form on $M$. In what follows, we will denote with
 $$
  C^{k,\alpha}_0(M)=\left\{v\in C^{k,\alpha}(M)\quad  \middle| \quad  \int_MvdV=0\right\}\,,\quad \forall \alpha\in (0,1)\, ,\,\,\,\forall k\ge 0\,.
  $$ 

First of all, we can observe easily from  equation (\ref{eq:equazione}) that $AB>0$. Then $A$ and $B$ have the same sign. On the other hand, if $p\in M$ is the point where $u$ attains its minimum, we have $\nabla u=0$  and $u_{nn}>0$ at $p$. So, $A,B>0$  on $M$. This, together with the fact that
\begin{equation}\label{eq:5}
A+B=2+\Delta u+(F+G)(\nabla u )>0\,, 
\end{equation} gives a lower bound on $Lu=\Delta u+(F+G)(\nabla u )$. Obviously, $L$ is a linear elliptic differential operator of second order.  Furthermore, this lower bound can be easily improved.
\begin{lem}\label{lem:2.1}
Let $F$ and $G$ be two smooth functions such that $F(0)=0, G(0)=0$ and let $u\in C^2_0(M)$ a solution of (\ref{eq:3}).  Then, the following holds: \begin{equation}\label{eq:31}
A+B\ge 2e^{\frac{f}{2}}\,.
\end{equation}
\end{lem}
 \begin{proof}From (\ref{eq:equazione}), we obtain that
$$
AB\ge e^{f}\,.$$ On the other hand, we know that $A^2+B^2\ge 2AB. $ So, $$
(A+B)^2\ge 4AB\ge 4e^{f}\,.
$$
From this, the assertion follows.
\end{proof} 
%Lemma \ref{lem:2.1} guarantees a uniform estimate on $b$. Indeed, if $q\in M$ is the point where $u$ attains its minimum, we know that $D^2u(q)\ge 0$ and $\nabla u(q)=0$. In particular, this guarantees that $u_{in}^2\le u_{ii}u_{nn}$ at $q$. Then, \[e^{f+b}\ge 1+\Delta u\ge 1\,\,\,\,\,\, \mbox{at } q.  \] This implies $b\ge -\lVert f \rVert_{C^0(M)}.$ On the other hand, thanks to Lemma \ref{lem:2.1}, we have \[e^{\frac{b}{2}}\le \frac{2+\Delta u+(F+G)(\nabla u)}{2e^{\frac{f}{2}}},\] at any point. Choosing $p\in M$ as the point where $u$ attains its  maximum, we know that $\nabla u(p)=0$ and $\Delta u(p)\le 0$. Then,  $b\le \lVert f\rVert_{C^0(M)}$. These two inequalities guarantee that $\lvert b\rvert\le C$. Moreover,
 Lemma \ref{lem:2.1} will be used in order to prove the $C^2$ estimate in Section 4.

 The next Lemma is just a technical result which will be extremely useful in order to prove the ellipticity of equation (\ref{eq:3}).
 \begin{lem}\label{lem:1.2}
 Let $a,b,c_i\in \R$, $\forall i=1,\ldots, n-1$. Then, the characteristic polynomial of
 $$
 P_n=\begin{pmatrix}a & 0 & \cdots & 0 &-c_1 \\ 0 & a & 0 & \cdots &-c_2 \\   \vdots &  & \ddots & & \vdots \\0 & \cdots & \cdots & a & -c_{n-1} \\ -c_1& -c_2 &\cdots & -c_{n-1}&  b \end{pmatrix} $$
  is $$
  \det(P_n-\lambda Id)=(a-\lambda)^{n-2}\left(\lambda^2-(a+b)\lambda+ab-\sum_{i=1}^{n-1}c_i^2\right)\,.$$
  \end{lem}
   Lemma \ref{lem:1.2} can be proved by a straightforward induction. Then, ellipticity of (\ref{eq:3}) is a direct consequence of Lemma \ref{lem:1.2}.
   \begin{prop}\label{prop:1.3} 
   Equation (\ref{eq:3}) is elliptic. Moreover, $\forall \zeta\in \R^n$, $\zeta\ne 0 $, we have
   \begin{equation}\label{eq:4}
   A\sum_{i=1}^{n-1}\zeta_i^2+B\zeta_n^2-2\sum_{i=1}^{n-1}u_{in}\zeta_i\zeta_n\ge \lambda_-\lvert\zeta \rvert^2\,,
   \end{equation} where $$
   \lambda_-=\frac{1}{2}\left(A+B-\sqrt{(A+B)^2-4e^{f}}\right)\,.
   $$
   \end{prop}
   \begin{proof}
   Define the operator $\Phi\colon C^{2}(M)\to C^0(M)$ such that, $\forall u\in C^2(M),$ $$
   \Phi(u)=AB-\sum_{i=1}^{n-1}u_{in}^2\,.
   $$
    Suppose that $u$ is a solution for (\ref{eq:3}), then  the linearization of $\Phi$ at $u$ computed in $v\in C^2(M)$ is:
    $$
    d_u\Phi(v)=Bv_{55}+A\left(\sum_{i=1}^{n-1}v_{ii}\right)-2\sum_{i=1}^{n-1}u_{in}v_{in}+A\frac{d}{dt}F(\nabla(u+tv))_{|_{t=0}}+B\frac{d}{dt}G(\nabla (u+tv))_{|_{t=0}}\,.
    $$
     Then, the matrix corresponding to the principal symbol of $\Phi$ is 
     $$
     P_u=\begin{pmatrix}A & 0 & \cdots & 0 &-u_{1n} \\ 0 & A & 0 & \cdots &-u_{2n} \\   \vdots &  & \ddots & & \vdots \\0 & \cdots & \cdots & A & -u_{n-1n} \\ -u_{1n}& -u_{2n} &\cdots & -u_{n-1n}&  B \end{pmatrix}
     $$
      whose eigenvalues, thanks to Lemma \ref{lem:1.2}, are$$
      \lambda=A\,, \quad  \lambda_{\pm}=\frac{1}{2}\left(A+B\pm\sqrt{(A+B)^2-4e^{f}}\right)>0\,.
      $$
       It is easy to prove that $\lambda_{-}\le A\le \lambda_{+}$. This guarantees the assertion.
       \end{proof}
\section{$C^0$ estimate.}

The method we used to obtain the $C^0$ estimate on the solution is based on  Sz\'ekelyhidi's method in \cite{9}.  In this section, we will suppose that $F(\nabla v)=X^iv_i$ and $G(\nabla v)=Y^iv_i$, $\forall v\in C^{1}(M)$, where $X,Y$ are smooth vector fields on $M$, without any further hypothesis. Then, the estimate we proved holds true in a more general setting than the one described by Theorem \ref{thm:1.1}.

Let $u\in C^2_0(M)$ a solution of (\ref{eq:3}) and  $p\in M$ be the point where $u $ attains its minimum. Consider a local chart centered in $p$ and suppose that its domain is $B(0,1)\subset \mathbb{R}^n$. Fix $\varepsilon >0$ and define the function $$
\varphi(x)=u(x)-\max_Mu+\varepsilon\lvert x\rvert^2\,,\quad \forall x\in B(0,1)\,.
$$
Clearly, we have that$$ D^2\varphi =D^2u+2\varepsilon Id\,,
\quad  \nabla\varphi=\nabla u+2\varepsilon x\,.
$$
Then, $\varphi$ is a solution of the following:
\begin{equation}\label{eq:7}
\left(1-2\varepsilon+\varphi_{nn}+G(\nabla \varphi)- 2\varepsilon G(x)\right)\left(1-2(n-1)\varepsilon+\sum_{i=1}^{n-1}\varphi_{ii}+F(\nabla \varphi)-2\varepsilon F(x)\right)-\sum_{i=1}^{n-1}\varphi_{in}^2=e^{f}\,.
\end{equation} 
Define $$
P=\{x\in B(0,1) \quad  |\quad \lvert \nabla\varphi(x)\rvert< \frac{\varepsilon}{2}\,,\,\,\varphi(y)\ge \varphi(x)+\nabla\varphi(x)\cdot (y-x)\,,\,\,\forall y\in B(0,1) \}\,,
$$
 called the contact set of $\varphi$. Then, easily, we see that $$
 \varphi(x)\le \varphi(0)+\nabla\varphi(x)\cdot x\le \varphi(0)+\frac{\varepsilon}{2}=\min_Mu-\max_Mu+\frac{\varepsilon}{2}\,,\quad  \forall x\in P
 $$
  which implies that 
  \begin{equation}\label{eq:8}
  \tilde{u}(x)\le \min_M\tilde{u}+\frac{\varepsilon}{2}\,,\quad \forall x\in P\,,
  \end{equation} where $$
  \tilde{u}=u-\max_Mu\le 0\,.
  $$
   So, given any $p\in[1,\frac{n}{n-2})$, we can apply the function $h(t)=-t^p$  to (\ref{eq:8})  and integrate it on $P$ obtaining that:  
   \begin{equation}\label{eq:9}
   \lVert u\rVert_{C^0(M)}\le \max_Mu-\min_M u=\lVert \tilde{u}\rVert_{C^0(M)}\le \frac{\lVert \tilde{u}\rVert_{L^p(M)}}{Vol(P)^{\frac{1}{p}}}+\frac{\varepsilon}{2}\,.
   \end{equation}
   \begin{rmk}
   The inequality $$
   \lVert u\rVert_{C^0(M)}\le \max_Mu-\min_M u
   $$
    is a direct consequence of the fact that $u$ has zero mean value on $M$.\end{rmk} Then, in order to obtain the estimate we want, it is sufficient to find a  uniform lower bound on $Vol(P)$ and a uniform upper bound on $\lVert \tilde{u}\rVert_{L^p(M)}$. As regards the latter, we recall that, thanks to (\ref{eq:5}), $-\tilde{u}$ is a non-negative supersolution for the equation $Lv=-2$, where, in our hypothesis, $L$ is a linear elliptic  differential operator of second order with smooth coefficients. Then, we can use the following result.
    \begin{thm}[See \cite{7}, Theorem 8.18]\label{thm:2.1}
     Let $q>n$, $p\in [1, \frac{n}{n-2})$, and $g\in L^q(\R^n)$. Suppose that $L$ is a   strictly elliptic and linear differential operator of second order with bounded coefficients. Then, if $u\in W^{1,2}(\R^n)$ is a non-negative supersolution of $Lu=g$ in $B(y,4R)$, we have
     $$
     \lVert u\rVert_{L^p(B(y,2R))}\le C\left(\inf_{B(y, R)}u+ K(R)\right)\,,
     $$
      where $C>0$ is a uniform constant. \end{thm} 
      Applying Theorem \ref{thm:2.1} to $-\tilde{u}$, we obtain a local uniform upper bound on the $L^p$-norm of $\tilde{u}$. This can be extended to a global uniform upper bound on the $L^p$-norm of $\tilde{u}$ with a standard covering argument, see \cite[p.347]{9} for the details.

So, it remains to find a uniform lower bound on $Vol(P).$ First of all, we observe that, $$
\varphi(0)+\varepsilon=\min_Mu-\max_Mu+\varepsilon\le \min_{\partial B(0,1)}\varphi\, .
$$ Then, we can apply the following Proposition due to Sz\'ekelyhidi. \begin{prop}[\cite{9}, Proposition 10]\label{prop:2.3} 
Let $v\colon B(0,1)\to \R$ a smooth function such that $$
v(0)+\varepsilon \le \inf_{\partial B(0,1)}v\, .
$$ Then, there exists a dimensional constant $C>0$ such that $$
C\varepsilon^n\le \int_P\det(D^2v)\, ,
$$ where $P$ is the contact set of $v.$\end{prop} So, by applying Proposition \ref{prop:2.3}, we obtain the following estimate
$$
C\varepsilon^n\le \int_P\det(D^2\varphi)\, .
$$On the other hand, we know that, on $P$, $\varphi$ is convex. Thanks to this, $D^2\varphi (x)\ge0$, $\forall x\in P$. Then, applying the arithmetic-geometric mean inequality, we have
$$
\det(D^2\varphi)\le \left(\frac{\Delta u}{n}\right)^{n}\,,\quad  \mbox{ on } P\,.
$$ So, in order to conclude, it is sufficient to obtain a uniform upper bound on $\Delta u$ on $P$.

First of all, the fact that, $\forall x\in P$, $D^2u(x)\ge 0$ implies that $$
\varphi_{ii}(x)\ge 0\quad \mbox{and }\quad  \varphi_{ii}(x)\varphi_{jj}(x)-\varphi^2_{ij}(x)\ge 0\,,\quad  \forall x\in P\,,\quad \forall i,j=1,\ldots, n\,.$$ Moreover, on $P$, the following inequalities hold:
$$
\lvert F(\nabla\varphi)\rvert\le \lvert X\rvert\vert\nabla\varphi\rvert<\frac{\varepsilon}{2}\lVert X \rVert_{C^0(M)}\,, \quad  \lvert G(\nabla\varphi)\rvert\le \lvert Y\rvert\vert\nabla\varphi\rvert<\frac{\varepsilon}{2}\lVert Y \rVert_{C^0(M)}\,.
$$
 Choosing  $$
 \varepsilon <\max\left\{\frac{1}{2(n-1)+\frac{5}{2}\lVert X \rVert_{C^0(M)}}, \frac{1}{2+\frac{5}{2}\lVert Y \rVert_{C^0(M)}}\right\}\,,
 $$
 we have
 $$
 \begin{aligned}
 e^{f}\ge&\,  \left(1-2\varepsilon+\varphi_{nn}+G(\nabla \varphi)-2\varepsilon G(x)\right)\left(1-2(n-1)\varepsilon+\sum_{i=1}^{n-1}\varphi_{ii}+F(\nabla\varphi)-2\varepsilon F(x)\right)-\varphi_{nn}\sum_{i=1}^{n-1}\varphi_{ii}\\
 \ge&\,  \left(1-\varepsilon\left(2+\frac{5}{2}\lVert Y\rVert_{C^0(M)}\right)+\varphi_{nn}\right)\left(1-\varepsilon\left(2(n-1)+\frac{5}{2}\lVert X\rVert_{C^0(M)}\right)+\sum_{i=1}^{n-1}\varphi_{ii}\right)-\varphi_{nn}\sum_{i=1}^{n-1}\varphi_{ii}\\
 \ge&\,  \left(1-\varepsilon\left(2+\frac{5}{2}\lVert Y\rVert_{C^0(M)}\right)\right)\sum_{i=1}^{n-1}\varphi_{ii}+\left(1-\varepsilon\left(2(n-1)+\frac{5}{2}\lVert X\rVert_{C^0(M)}\right)\right)\varphi_{nn}\,.
 \end{aligned}
 $$ From this, we obtain that $\Delta u\le C$ on $P.$
 So, we proved the following Theorem.
 \begin{thm}\label{thm:2.4}
  Let $X$ and $Y$ be two smooth vector fields such that $$
  F(\nabla v)=X^iv_i\,,\quad G(\nabla v)=Y^iv_i\,,\quad \forall v\in C^1(M)\,,
  $$ and let $u\in C_0^2(M)$ be a solution of (\ref{eq:3}). Then, there exists a uniform  constant $C>0$ such that $$
  \lVert u\rVert_{C^0(M)}\le C\,.$$
   \end{thm}

\section{Higher order estimates.}

 In this section, we  prove higher order estimates.  First of all, we  prove a $C^0$ estimate on the Laplacian of $u$. Then, we  show that all the higher order estimates can be deduced from that one. The main Theorem of this section is the following.
 \begin{thm}\label{thm:3.1}
 In the same hypothesis of Theorem \ref{thm:1.1}, let $u\in C_0^{4}(M)$ be a solution of (\ref{eq:3}). Then, there exists a  uniform constant $C>0$ such that \begin{equation}\label{eq:10}
 \lVert\Delta u\rVert_{C^0(M)}\le C(1+\lVert u\rVert_{C^1(M)})\,.
 \end{equation}\end{thm} 

Before starting the proof of Theorem \ref{thm:3.1}, 
 we should observe that Theorem \ref{thm:1.1} can be considered as the generalization of both \cite[Theorem 10]{2}  and \cite[Theorem 13]{6}. Indeed, as regards the first case, we have that $Y=0$ and $$
X=\begin{pmatrix} 0 \\ 0 \\ 1\end{pmatrix}\,,
$$ while, in the second one, $X=0=Y.$

The following Lemma is a slight generalization of  \cite[Lemma 7]{6}. The proof is the same as in \cite{6} but, for completeness, we  briefly discuss it.\begin{lem}\label{lem:3.2}
Let $\varepsilon \in \R$, $g\in C^{2}(M)$ and $p_0\in M $ be the point where $$
\psi=ge^{-\varepsilon u}
$$ attains its maximum. We define $$
\eta_{ij}=\varepsilon g(u_{ij}+\varepsilon u_iu_j)-g_{ij}\,,\quad  \forall i,j=1,\ldots, n\,.
$$
 Then, $$
 \eta_{ii}(p_0)\ge 0\,,\quad  \sqrt{\eta_{ii}\eta_{jj}}\ge \lvert\eta_{ij}\rvert \quad  \mbox{ at } p_0\,.
 $$
 \end{lem}
 \begin{proof} 
 We have $$
 \nabla \psi=e^{-\varepsilon u}\left(\nabla g-\varepsilon g\nabla u\right)\,.
 $$
  At $p_0$, we know that $\nabla\psi=0$ which implies 
  \begin{equation}\label{eq:24}
  \nabla g=\varepsilon g\nabla u\,.
  \end{equation} On the other hand, at $p_0$, using (\ref{eq:24}), we obtain that $$
  \begin{aligned}
  \psi_{ij}=&\, \varepsilon^2e^{-\varepsilon u}gu_iu_j-\varepsilon e^{-\varepsilon u}(u_ig_j+ g_iu_j+gu_{ij})+e^{-\varepsilon u}g_{ij}\\
  =&\, \varepsilon e^{-\varepsilon u} g_iu_j-\varepsilon e^{-\varepsilon u}(u_ig_j+ g_iu_j+gu_{ij})+ e^{-\varepsilon u} g_{ij}\\
  =&\, e^{-\varepsilon u} g_{ij}-\varepsilon e^{-\varepsilon u}(u_ig_j+gu_{ij})=e^{-\varepsilon u}(g_{ij}-\varepsilon g(u_{ij} +\varepsilon u_iu_j))\,.
  \end{aligned}
  $$
   So,$$
   D^2\psi=e^{-\varepsilon u}\left(D^2g-\varepsilon g(D^2u+\varepsilon \nabla u \otimes\nabla u)\right)\,.
  $$ Then, at $p_0$, $$
D^2\psi\le 0
$$ 
which implies
 $$
\varepsilon g(D^2u+\varepsilon \nabla u \otimes\nabla u)-D^2g\ge 0\,.   
$$ 
From this, the assertion follows.
\end{proof} 
The technique we used to prove Theorem \ref{thm:3.1} is an adaptation of  the one used in \cite{2} and in \cite{6}.
% In the following, we will denote with 
%$$
%F_k(\nabla v)=\frac{\partial X^i}{\partial x_k}v_i\,,\quad F_{ks}(\nabla v)=\frac{\partial^2X^i}{\partial x_k\partial x_s}v_i\,,\quad \forall k,s=1,\ldots, n\,,\quad \forall v\in C^{\infty}(M)\,.
%$$
% With these notations, we have $$
 %\frac{\partial }{\partial x_k}F(\nabla v)=F_k(\nabla v)+F(\nabla v_k)\,,\quad  \frac{\partial^2}{\partial x_k\partial x_s}F(\nabla v)=F_{ks}(\nabla v)+F_k(\nabla v_s )+F_s(\nabla v_k)+F(\nabla v_{ks})
 %$$
 % and hypothesis \ref{H:1}) can be rewritten as follows:
  %\begin{equation}\label{eq:12}
%G_k=0\,,  \quad F_n=0\,,\quad  F_{kn}=0\,,\quad  \forall k=1,\ldots, n\,.\end{equation} 
\begin{proof}[ Proof of Theorem \ref{thm:3.1}]
Easily, we see that
 \begin{equation}\label{eq:13}
 \begin{aligned}\Delta e^{f}+(F+G)(\nabla e^{f})=&\, A(\Delta B+(F+G)(\nabla B))+B(\Delta A+ (F+G)(\nabla A)) \\ +&\, 2 \nabla A\cdot \nabla B - 2\sum_{i=1}^{n-1}\left(u_{in}(\Delta u_{in}+(F+G)(\nabla u_{in}))+\lvert\nabla u_{in}\rvert^2\right)\,.\end{aligned}
 \end{equation}
 By straightforward calculations, we can observe that 
 \begin{equation}\label{eq:14}
  \begin{aligned}\Delta B+(F+G)(\nabla B)=&\, \sum_{i=1}^{n-1}\left(\Delta u_{ii}+(F+G)(\nabla u_{ii})\right) +(F+G)(\nabla F(\nabla u)) 
  +\sum_{j=1}^nF(\nabla u_{jj}) 
  %+2F_j(\nabla u_j)+F_{jj}(\nabla u)\right)
  \end{aligned}
  \end{equation}
  and
   \begin{equation}\label{eq:15}
    \Delta A+ (F+G)(\nabla A)=\Delta u_{nn}+ (F+G)(\nabla u_{nn})+\sum_{j=1}^nG(\nabla u_{jj})+ (F+G)(\nabla G(\nabla u))\,.
  \end{equation} Consider $g=A+B$ and $\varepsilon>0 $ that will be determined later. From now on, unless otherwise stated, all the quantities and inequalities will be computed at $p_0\in M$ as in Lemma \ref{lem:3.2}. Then, we can apply (\ref{eq:4}) choosing $\zeta_i=sign(u_{in})\sqrt{\eta_{ii}}, \,\, \forall i=1,\ldots,n-1$ and $\zeta_n=\sqrt{\eta_{nn}}$, obtaining that 
  \begin{equation}
  \label{eq:56} 0\le A\sum_{i=1}^{n-1}\eta_{ii}+B\eta_{nn}-2\sum_{i=1}^{n-1}\lvert u_{in}\rvert\sqrt{\eta_{ii}\eta_{nn}}\,.
  \end{equation}
  Applying Lemma \ref{lem:3.2}, we obtain
   $$
  0\le A\sum_{i=1}^{n-1}\eta_{ii}+B\eta_{nn}-2\sum_{i=1}^{n-1}\lvert u_{in}\rvert\lvert\eta_{in}\rvert \le  A\sum_{i=1}^{n-1}\eta_{ii}+B\eta_{nn}-2\sum_{i=1}^{n-1} u_{in}\eta_{in}\,.
  $$
   Now, using the definition of $\eta_{ij}$'s, we have that
   \begin{equation}\label{eq:20}\begin{aligned}
   0\le&\,  \varepsilon gA\sum_{i=1}^{n-1}\left(u_{ii}+\varepsilon u^2_i\right)+\varepsilon gB(u_{nn}+\varepsilon u_n^2)-2\varepsilon g\sum_{i=1}^{n-1}u_{in}(u_{in}+\varepsilon u_iu_n)
   - A\sum_{i=1}^{n-1}g_{ii}-Bg_{nn}+2\sum_{i=1}^{n-1}u_{in}g_{in}.\end{aligned}
   \end{equation} 
   On the other hand, we notice that
   \begin{equation}\label{eq:43}
   g_{ij}=\Delta u_{ij}+(F+G)(\nabla u_{ij})\,.
   %+F_{ij}(\nabla u)+F_i(\nabla u_j)+F_j(\nabla u_i)\,,\quad  \forall i,j=1,\ldots, n\,.
   \end{equation} 
   Substituting  (\ref{eq:43}), we obtain that 
   \begin{equation}\label{eq:16}\begin{aligned}
   -A \sum_{i=1}^{n-1}g_{ii}-Bg_{nn}+2\sum_{i=1}^{n-1}u_{in}g_{in}=&\,- A\sum_{i=1}^{n-1}\left(\Delta u_{ii}+(F+G)(\nabla u_{ii})\right)- B(\Delta u_{nn} + (F+G)(\nabla u_{nn}))\\+&\, 2\sum_{i=1}^{n-1}u_{in}(\Delta u_{in}+(F+G)(\nabla u_{in}))\,.  
   %-A\left(\sum_{i=1}^{n-1}\left(2F_i(\nabla u_i)+F_{ii}(\nabla u)\right)\right)+2\sum_{i=1}^{n-1}u_{in}F_i(\nabla u_n)\,.
   \end{aligned}
   \end{equation}
    Substituting (\ref{eq:14}) and (\ref{eq:15}) in (\ref{eq:16}), it holds that 
     \begin{equation}\label{eq:17} 
     \begin{aligned}
     -&\,A \sum_{i=1}^{n-1}g_{ii}-Bg_{nn}+2\sum_{i=1}^{n-1}u_{in}g_{in}=- A(\Delta B+(F+G)(\nabla B))
      -  B(\Delta A + (F+G)(\nabla A))\\+&\,2\sum_{i=1}^{n-1}u_{in}(\Delta u_{in}+(F+G)(\nabla u_{in})) 
      %+2\sum_{i=1}^{n-1}u_{in}F_i(\nabla u_n) \\
+ A\left(\Delta F(\nabla u)+ (F+G)(\nabla F(\nabla u))
       %-\sum_{i=1}^{n-1}2F_i(\nabla u_i)+F_{ii}(\nabla u)
       \right)\\
       +&\, B(\Delta G(\nabla u)+ (F+G)(\nabla G(\nabla u)))\,.
     \end{aligned}
     \end{equation}
      On the other hand, by straightforward calculation, we see that 
      \begin{equation}\label{eq:18}
      \Delta F(\nabla u)+ (F+G)(\nabla F(\nabla u))
      %-\sum_{i=1}^{n-1}2F_i(\nabla u_i)+F_{ii}(\nabla u)
      =F(\nabla g)\end{equation}and \begin{equation}\label{eq:19} \Delta G(\nabla u)+ (F+G)(\nabla G(\nabla u)))=G(\nabla g)\,.
      \end{equation} 
      Then, substituting (\ref{eq:18}) and (\ref{eq:19}) in (\ref{eq:17}) and substituting the result in (\ref{eq:20}), we obtain 
      \begin{equation}\label{eq:21}
      \begin{aligned}
      &\,A(\Delta B+(F+G)(\nabla B))+  B(\Delta A + (F+G)(\nabla A))-2\sum_{i=1}^{n-1}u_{in}(\Delta u_{in}+(F+G)(\nabla u_{in})) \\
       \le&\, \varepsilon gA\sum_{i=1}^{n-1}(u_{ii}+\varepsilon u^2_i)+\varepsilon gB(u_{nn}+\varepsilon u_n^2)-2\varepsilon g\sum_{i=1}^{n-1}u_{in}(u_{in}+\varepsilon u_iu_n)
       +AF(\nabla g)+ BG(\nabla g)\,.
       %+2\sum_{i=1}^{n-1}u_{in}F_i(\nabla u_n)\,.
       \end{aligned}
       \end{equation}
       This inequality can be substituted in (\ref{eq:13}) obtaining that
       \begin{equation}\label{eq:22}
       \begin{aligned}
       \Delta e^{f}+ (F+G)(\nabla e^{f})\le&\,\varepsilon gA\sum_{i=1}^{n-1}(u_{ii}+\varepsilon u^2_i)+\varepsilon gB(u_{nn}+\varepsilon u_n^2)-2\varepsilon g\sum_{i=1}^{n-1}u_{in}(u_{in}+\varepsilon u_iu_n)\\
       +&\, 2\nabla A\cdot \nabla B+ AF(\nabla g)+ BG(\nabla g)\,.
       %+2\sum_{i=1}^{n-1}u_{in}F_i(\nabla u_n)\,.
       \end{aligned}
      \end{equation}
       %On the other hand, we observe that 
      %\begin{equation}\label{eq:23}
      %\sum_{i=1}^{n-1}u_{in}F_i(\nabla u_n)=\sum_{i=1}^{n-1}\sum_{j=1}^{n}u_{in}\frac{\partial X^j}{\partial x_i}u_{jn}=\nabla u_n\cdot \frac{\partial X}{\partial x}\nabla u_n\le 0\,, 
      %\end{equation}
       %thanks to hypothesis \ref{H:2}).
        Moreover, thanks to Lemma \ref{lem:3.2} and to (\ref{eq:24}), at $p_0$, we have $$
       \lvert\nabla g\rvert^2=\varepsilon^2g^2\lvert \nabla u\rvert^2\,.
       $$
        On the other hand,
        \begin{equation}\label{eq:26}
        \varepsilon^2g^2\lvert \nabla u\rvert^2=\lvert\nabla g\rvert^2=\lvert\nabla (A+B)\rvert^2=\lvert\nabla A\rvert^2+\lvert\nabla B\rvert^2+2\nabla A\cdot\nabla B\ge2\nabla A\cdot\nabla B \,.
        \end{equation}
        Moreover, at $p_0$,  since (\ref{eq:24}) holds true, we have  
        \begin{equation}\label{eq:27}
        F(\nabla g)=\varepsilon g F(\nabla u)\,,\quad G(\nabla g)=\varepsilon g G(\nabla u)\,.
        \end{equation}
        Then, using  (\ref{eq:26}) and (\ref{eq:27}), we can obtain the following inequality
   $$
   \begin{aligned}
        &\, \Delta e^{f}+ (F+G)(\nabla e^{f})\le\varepsilon gA\sum_{i=1}^{n-1}(u_{ii}+\varepsilon u^2_i)+\varepsilon gB(u_{nn}+\varepsilon u_n^2) \\ 
         -&\, 2\varepsilon g\sum_{i=1}^{n-1}u_{in}(u_{in}+\varepsilon u_iu_n) +\varepsilon^2g^2\lvert \nabla u\rvert^2+ A\varepsilon g F(\nabla u)+ B\varepsilon gG(\nabla u) \\
=&\, \varepsilon g\left(A\left(B-1\right)+B(A-1)-2\sum_{i=1}^{n-1}u^2_{in}\right)+\varepsilon^2g\left(A\sum_{i=1}^{n-1}u_i^2+Bu_n^2-2\sum_{i=1}^{n-1}u_{in}u_iu_n\right) +\varepsilon^2g^2\lvert\nabla u\rvert^2\,. 
\end{aligned}
$$
 From the inequality above, using equation (\ref{eq:equazione}), we obtain
 \begin{equation}\label{eq:28}
 \Delta e^{f}+(F+G)(\nabla e^{f})\le 2\varepsilon ge^{f}-\varepsilon g^2+\varepsilon^2g\left(A\sum_{i=1}^{n-1}u_i^2+Bu_n^2-2\sum_{i=1}^{n-1}u_{in}u_iu_n\right)+\varepsilon^2g^2\lvert\nabla u\rvert^2\,.
 \end{equation}
 Choosing $\zeta_i=u_i,\,\,\,\forall i=1,\ldots, n-1$,  and $\zeta_n=-u_n$, thanks to (\ref{eq:4}), we easily obtain that 
 \begin{equation}\label{eq:29}
 -2\sum_{i=1}^{n-1}u_{in}u_iu_n\le A\sum_{i=1}^{n-1}u_i^2+Bu_n^2\le g\lvert\nabla u\rvert^2\,.
 \end{equation} 
 Then, substituting (\ref{eq:29}) in (\ref{eq:28}), we have $$
 \Delta e^{f}+(F+G)(\nabla e^{f})\le2\varepsilon ge^{f}-\varepsilon g^2+2\varepsilon^2g^2\lvert\nabla u\rvert^2\,,
 $$ which is equivalent to
  \begin{equation}\label{eq:30}
  \varepsilon g^2\le-(\Delta e^{f}+ (F+G)(\nabla e^{f}))+ 2\varepsilon ge^{f}+3\varepsilon^2g^2\lvert\nabla u\rvert^2\,.
  \end{equation}
   So, we obtain that 
   \begin{equation}\label{eq:34}
   \varepsilon g^2\le\lVert(\Delta e^{f}+(F+G)(\nabla e^{f}))\rVert_{C^0(M)}+ 2\varepsilon g\lVert e^{f}\rVert_{C^0(M)}+3\varepsilon^2g^2\lvert\nabla u\rvert^2\,.
   \end{equation} 

Consider, now, $p_1\in M$ to be the point where $g$ attains its maximum, $\psi$ as in Lemma \ref{lem:3.2} and $$
\varepsilon=\frac{1}{g(p_1)}\,.
$$
 Then, we have $$
 g(p_0)\le g(p_1)=\psi(p_1)e^{\varepsilon u(p_1)}\le \max_M\psi e^{\varepsilon u(p_1)}=g(p_0)e^{\varepsilon(u(p_1)-u(p_0))}\le g(p_0) e^{2\varepsilon \lVert u\rVert_{C^0(M)}}\,.
 $$
 Moreover, using (\ref{eq:31}), we observe that $$
 2\varepsilon\le \frac{1}{e^{\frac{f}{2}}}\le e^{-\min_M\frac{f}{2}}\,.
 $$
 Furthermore, we notice that
  \begin{equation}\label{eq:32}
  \exp\left(-e^{-\min_M\frac{f}{2}}\lVert u\rVert_{C^0(M)}\right)\le \exp(-2\varepsilon\lVert u\rVert_{C^0(M)})=\varepsilon g(p_1)e^{-2\varepsilon\lVert u\rVert_{C^0(M)}}\le \varepsilon g(p_0)\,.
  \end{equation} 
  Multiplying (\ref{eq:32}) by $g(p_1)$, we obtain that

\begin{equation}\label{eq:33}
\exp\left(-e^{-\min_M\frac{f}{2}}\lVert u\rVert_{C^0(M)}\right)g(p_1)
\le g(p_0)\,.
\end{equation}
Again, multiplying (\ref{eq:32}) and (\ref{eq:33}), we have that
$$
 \exp\left(-2e^{-\min_M\frac{f}{2}}\lVert u\rVert_{C^0(M)}\right)g(p_1)\le \varepsilon g(p_0)^2\,,
 $$
  which, thanks to (\ref{eq:34}) and observing that $\varepsilon g(p_0)\le 1$, guarantees the following
  \begin{equation}\label{eq:35} 
  \exp\left(-2e^{-\min_M\frac{f}{2}}\lVert u\rVert_{C^0(M)}\right)g(p_1)\le  \lVert(\Delta e^{f}+(F+G)(\nabla e^{f}))\rVert_{C^0(M)}+ 2 \lVert e^{f}\rVert_{C^0(M)}+3\lVert\nabla u\rVert_{C^0(M)}\,.
  \end{equation} 
  From (\ref{eq:35}), it is easy, using Theorem \ref{thm:2.4}, to deduce a uniform upper bound for $g$ as follows
  \begin{equation}\label{eq:36}
  2+\Delta u +(F+G)(\nabla u )=g\le C(1+\lVert u\rVert_{C^1(M)})\,.\end{equation}
  In order to conclude, it is sufficient to recall that $$
  \lvert F(\nabla u)\rvert\le \lvert X\rvert\lVert u\rVert_{C^{1}(M)}\,,\quad \lvert G(\nabla u)\rvert\le \lvert Y\rvert\lVert u\rVert_{C^{1}(M)}\,.
  $$
   From these and (\ref{eq:36}), the assertion follows.

\end{proof}
\begin{rmk}\label{cutto}
The argument used in the proof of Theorem \ref{thm:3.1} works  assuming slightly weaker hypothesis, which are the following: $Y$ is constant and X does not depend on $x_n$, the matrix $\frac{\partial X}{\partial x}=(\frac{\partial X^i}{\partial x_j})_{i,j}$ is negative semidefinite and   $$
0=\sum_{i=1}^{n-1}Y^i\frac{\partial X^j}{\partial x_i}\,, \quad \forall j=1,\ldots, n-1\,.
$$ However, assuming that $\frac{\partial X}{\partial x}$ is negative semidefinite immediately guarantees that $X$ is constant.  Indeed, thanks to the Sylvester's criterion, we have that 
\begin{equation}\label{groda}
\frac{\partial X^i}{\partial x_i}\le 0\,, \quad \frac{\partial X^i}{\partial x_i}\frac{\partial X^j}{\partial x_j}- \left(\frac{\partial X^j}{\partial x_i}\right)^2\ge 0\,, \quad \forall i,j=1,\ldots, n\,.
\end{equation}  From the first condition in \eqref{groda}, we deduce that, for all $i=1,\ldots,n$, the function $X^i$, as a function of $x_i$ only, is a $1$-periodic function which happens to be non-increasing, that is possible if and only if $X^i$ is constant with respect to $x_i$. This guarantees that 
$$
\frac{\partial X^i}{\partial x_i}=0\,,\quad \forall i=1,\ldots, n\,.
$$ Using this in the second condition in \eqref{groda}, we obtain that 
$$
\frac{\partial X^j}{\partial x_i}=0\,, \quad \forall i,j=1,\ldots, n\,,
$$ giving us the claim. 
\end{rmk}
Using the estimate (\ref{eq:10}), we succeed to prove the $C^1$ estimate. The technique is analogous to the one used both in \cite[Theorem 9]{6} and \cite[Theorem 7]{2}. However, for completeness, we give the details of the proof.
\begin{prop}\label{prop:3.3}
In the same hypothesis of Theorem \ref{thm:1.1}, let $u\in C_0^{4}(M)$ be a solution of equation (\ref{eq:3}). Then, there exists a uniform constant $C>0$ such that 
\begin{equation}\label{eq:37}
\lVert u\rVert_{C^1(M)}\le C\,.
\end{equation}
\end{prop}
\begin{proof}
We fix $p>n$ and, thanks to the Morrey's inequality for $W^{2,p}(M)$, we have that there exists a constant $C$,  depending only on $M$, such that, for a certain $\alpha\in (0,1)$, $$
\lVert u\rVert_{C^{1,\alpha}(M)}\le C\lVert u\rVert_{W^{2,p}(M)}\,.
$$ On the other hand, thanks to \cite[Theorem 9.11]{7}, there exists a constant $C>0$ such that 
$$
\lVert u\rVert_{W^{2,p}(M)}\le C(\lVert u\rVert_{L^p(M)}+\lVert \Delta u\rVert_{L^p(M)})\le C(\lVert u\rVert_{C^0(M)}+\lVert \Delta u\rVert_{C^0(M)})\le C(1+\lVert u\rVert_{C^1(M)})\,.
$$
 Using the standard interpolation theory, see \cite[Section 6.8]{7},  we know that, $\forall \varepsilon >0$,  there exists $P_{\varepsilon}>0$ such that
 $$
 \lVert u\rVert_{C^1(M)}\le P_{\varepsilon }\lVert u\rVert_{C^0(M)}+\varepsilon \lVert u\rVert_{C^{1,\alpha}(M)}\,.
 $$Then, we obtain $$
 \lVert u\rVert_{C^1(M)}\le CP_{\varepsilon}+\varepsilon C'(1+\lVert u\rVert_{C^1(M)})\,.
 $$ Choosing $\varepsilon<\frac{1}{C'}$, the claim follows.
\end{proof}
 As a direct corollary of Proposition \ref{prop:3.3} and Theorem \ref{thm:3.1}, we obtain the uniform $C^0$ bound on the Laplacian of $u$. 
 \begin{thm}\label{thm:3.4}
 In the same hypothesis of Theorem \ref{thm:1.1}, let $u\in C_0^{4}(M)$ be a solution of equation (\ref{eq:3}). Then, there exists a uniform constant $C>0$ such that 
 \begin{equation}\label{eq:37}
 \lVert \Delta u\rVert_{C^0(M)}\le C\,.
 \end{equation}
 \end{thm}
  Finally, the $C^{2,\alpha}$ estimate can be deduced, in a standard way, by \cite[Theorem 1.1]{11} using the same argument as in \cite{10}. An important condition to obtain  the $C^{2,\alpha}$ estimate is the concavity of the equation with respect to the second order derivatives of $u$. In order to prove this, first of all,  as done before,  we define  the quantities:
  $$
  A(T,Z)=1+T_{nn}+G(Z)\,, \quad B(T,Z)=1+\sum_{i=1}^{n-1}T_{ii}+F(Z)\,,\quad \forall (T,Z)\in  C^{\infty}(M, {\rm Sym}^2T^*M)\times C^{\infty}(M, TM)\,,
  $$where $C^{\infty}(M, {\rm Sym}^2T^*M)$ and  $C^{\infty}(M, TM)$ are, respectively,  the set of  smooth symmetric $2$-tensors on $M$ and the set of smooth vector fields on $M$. Moreover, we define the set 
  $$ \Gamma=\left\{ (T,Z)\in  C^{\infty}(M, {\rm Sym}^2T^*M)\times C^{\infty}(M, TM)\,\, \middle|\,\, A(T,Z)B(T,Z)-\sum_{i=1}^{n-1}T_{in}^2>0\right\}\,. $$ 
  We notice that $\Gamma$ is convex fiberwise, i.e. fixed $Z\in C^{\infty}(M, TM)$, for all $T,S\in C^{\infty}(M, {\rm Sym}^2T^*M)$ such that $(A,Z), (B, Z)\in \Gamma$ we have that $(tT+(1-t)S, Z)\in \Gamma$, $\forall t\in [0,1] $.  Indeed, we observe that 
  \begin{equation}\label{gigi}
  \begin{aligned}
  A&\, (tT+(1-t)S, Z)B(tT+(1-t)S, Z)- \sum_{i=1}^{n-1}(tT_{in}+(1-t)S_{in})^2\\
  >&\,  t(1-t)\left(A(S,Z)B(T,Z)+A(T,Z)B(S,Z)- 2\sum_{i=1}^{n-1}S_{in}T_{in}\right)\,
  \end{aligned}
 \end{equation} On the other hand, using Cauchy-Schwarz inequality, we have that 
  \begin{equation}\label{il}
  \sum_{i=1}^{n-1}S_{in}T_{in}\le \left(\sum_{i=1}^{n-1}S_{in}^2\right)^{\frac12}\left(\sum_{i=1}^{n-1}T_{in}^2\right)^{\frac12}< (A(T,Z)A(S,Z)B(T, Z)B(S, Z))^{\frac12}\,.
  \end{equation}
  Using \eqref{il} into \eqref{gigi}, we obtain that 
  $$
  \begin{aligned}A&\, (tT+(1-t)S, Z)B(tT+(1-t)S, Z)- \sum_{i=1}^{n-1}(tT_{in}+(1-t)S_{in})^2\\
  >&\, \left((A(S, Z)B(T,Z))^{\frac12}- (A(T,Z)B(S,Z))^{\frac12}\right)^2\ge 0\,, 
  \end{aligned}
  $$giving us the claim.
Now, we consider the following function: 
  $$
  \tilde{\Phi}(T,Z)=\log\left(A(T,Z)B(T,Z)-\sum_{i=1}^{n-1}T_{in}^2\right)- f\,,\quad \forall (T,Z)\in \Gamma\,.
  $$
 In order to prove the concavity of the function $\tilde \Phi$, we define another function:
 $$
 G(T,Z)=\begin{pmatrix}
 A(T,Z) & -\lvert v^T \rvert\\-\lvert v^T \rvert& B(T,Z)
 \end{pmatrix}\,,
 $$ where $v^T=(T_{1n},\ldots, T_{n-1n})$. Clearly, we have that 
 $$
 \tilde{\Phi}(T, Z)=\log\det G(T,Z)- f\,, \quad \forall (T,Z)\in \Gamma\,. 
 $$  Next,  we observe that, fixed $(T,Z),(S,Z)\in \Gamma$,  
 \begin{equation}\label{emma}
 \det( G(tT+(1-t)S, Z))\ge \det(tG(T,Z)+(1-t)G(S,Z))\,, \quad \forall t\in [0,1]\,.
 \end{equation}
   Indeed, we have that 
   $$
   \begin{aligned}
    \det &\, ( G(tT+(1-t)S, Z))-  \det(tG(T,Z)+(1-t)G(S,Z))\\
   =&\,  (t\lvert v^T\rvert+ (1-t)\lvert v^S\rvert)^2- \lvert v^{tT+(1-t)S}\rvert^2\ge 0\,.
    \end{aligned}
   $$
    Now,  applying  the function $x\mapsto x^{\frac12} $, which  is increasing, to \eqref{emma} and using the fact that $\det^{\frac12}$ is concave  on positive definite matrices,  we obtain 
    \begin{equation}\label{iala}
    (\det( G(tT+(1-t)S, Z)))^{\frac12}\ge \det(tG(T,Z)+(1-t)G(S,Z))^{\frac12}\ge t\det(G(T,Z))^{\frac12}+ (1-t)\det(G(S,Z))^{\frac12}\,.
    \end{equation} Finally, applying the function $\log$ to \eqref{iala}  and using the fact that it is concave, we obtain the claim.  
      
  Then, by a standard bootstrap argument, we proved the following. 
  \begin{thm}\label{thm:3.5}
   In the same hypothesis of Theorem \ref{thm:1.1}, let $u\in C_0^{4}(M)$ be a solution of equation (\ref{eq:3}). Then, $u\in C^{\infty}(M),$ and, $\forall k\ge 0$, there exists a uniform constant $C_k>0$ such that $$
   \lVert u\rVert_{C^k(M)}\le C_k\,.
   $$
   \end{thm}
\section{Proof of Theorem \ref{thm:1.1}.}
Once we obtained the uniform a priori estimates,  we are ready to prove Theorem \ref{thm:1.1}.\begin{proof}[Proof of Theorem \ref{thm:1.1}] Consider, $\forall t\in [0,1]$, the equation
$$
\Phi_t(u)=AB-\sum_{i=1}^{n-1}u^2_{in}-e^{f_t}=0\,,
$$
where $f_t=\log(1-t+te^{f})$ and define
$$
T=\{t\in [0,1]\quad  |\quad  \Phi_{t}(u)=0\,\quad \text{admits a solution  }u\in C_0^{2,\alpha}(M)\}\,.
$$ 
Obviously, $\Phi_0(u)=0$ admits a solution which is $u=0$. So, $T\ne\emptyset.$

 Then, fix $t\in T$ and consider $u\in C_0^{2,\alpha}(M)$ to be a solution of $\Phi_t(u)=0.$ Observe that, $\forall t\in [0,1]$,
 $$
 \Phi_t\colon V \to W
 $$ where $$
 V=\left\{v\in C^{2,\alpha}(M) \quad \middle|\quad  \int_MvdV=0\right\}=C_0^{2,\alpha}(M)\,,\quad W=C^{0,\alpha}(M)\,.
 $$
We have already computed the linearization at $u$ of $\Phi_t$ which is 
$$
d_u\Phi_t(v)=Bv_{nn}+ A\sum_{i=1}^{n-1}v_{ii}-2\sum_{i=1}^nu_{in}v_{in}+A F(\nabla v)+BG(\nabla v)\,,\quad \forall v\in T_uV\simeq V\,.
$$
  So, $d_u\Phi_t$ is a linear elliptic operator of second order without terms of order zero. Then, the strong maximum principle implies that $d_u\Phi_t$ is injective and, moreover, it has closed image. On the other hand, the symbol of $d_u\Phi_t$ is invertible and positive, thanks to Proposition \ref{prop:1.3}. Then, choosing $\{x_1,\ldots, x_n\}$ local coordinates, we can write
  $$
  d_u\Phi_t(v)=\Theta^{ij}v_{ij}+c^iv_i
  $$
  where  $c^i\in C^{\infty}(M),$ $\forall i =1,\ldots, n$ and $(\Theta^{ij})_{i,j}$ is positive, so, its inverse defines a riemannian metric on $M$. This implies that $$
  d_u\Phi_t(v)=\Delta_{\Theta}v+c^iv_i\,.
  $$
   In particular, the index of $d_uF_t$ coincides with that of $\Delta_{\Theta}$ which is zero. The injectivity of $d_u\Phi_t$ and this fact imply that $\ker((d_u\Phi_t)^*)=\{0\}$. We conclude observing that $$
    Im(d_u\Phi_t)=\overline{Im (d_u\Phi_t)}=\ker((d_u\Phi_t)^*)^{\bot}=C^{0,\alpha}(M)\,.
    $$
     So, $d_u\Phi_t$ is invertible. Applying the implicit function Theorem, we obtain that $T$ is open. Thanks to Theorem \ref{thm:3.5}, $T$ is also closed. From this, the existence of a solution follows. It remains to prove the uniqueness of the solution. Then, suppose that $u,v\in C_0^{\infty}(M)$ are two solutions of (\ref{eq:3}) and denote with $g\in C^{\infty}(M)$ the function such that $u=v+g.$ Clearly, since $\int_MudV=\int_M vdV=0$, then $\int_M gdV=0.$ By straightforward calculations, we notice that 
     $$
     e^{f}=(A_v+A_g-1)\left(B_v+B_g-1\right)-\sum_{i=1}^{n-1}(v_{in}^2+2v_{in}g_{in}+g_{in}^2)\,,
     $$
      which can be rewritten as follows:
      $$
      e^{f}=e^{f}+1+A_gB_g-\sum_{i=1}^{n-1}g_{in}^2-(A_g+B_g)+d_v\Phi(g)\,.
      $$
      Then, from this, we have that 
      \begin{equation}\label{eq:39}
      1+A_gB_g-\sum_{i=1}^{n-1}g_{in}^2-(A_g+B_g)+d_v\Phi(g)=0\,.
      \end{equation} 
      On the other hand, with the same calculations, we observe that 
      $$
      e^{f}=(A_u-A_g+1)\left(B_u-B_g+1\right)-\sum_{i=1}^{n-1}(u_{in}^2-2u_{in}g_{in}+g_{in}^2)\,,
      $$
       which is equivalent to 
       $$
       e^{f}=e^{f}+1+A_gB_g-\sum_{i=1}^{n-1}g_{in}^2-(A_g+B_g)-d_u\Phi(g)\,.
       $$
       Then, from the equation above, we obtain 
       \begin{equation}\label{eq:40}
        1+A_gB_g-\sum_{i=1}^{n-1}g_{in}^2-(A_g+B_g)-d_u\Phi(g)=0\,.
        \end{equation} 
        So, since $g$ has to solve both (\ref{eq:39}) and (\ref{eq:40}), it is a solution of 
        $$
        (d_u\Phi+d_v\Phi)(g)=0\,,
        $$ which is a linear elliptic  equation of second order without terms of order zero. Then, the strong maximum principle implies that $g$ must be constant. Since $\int_Mg dV=0$, $g=0$ and $u=v$. From this, the assertion follows.
        \end{proof}
\section{More general equations.}
Equation (\ref{eq:3}) is just a particular case of a more general class of elliptic fully nonlinear equations that can be described as follows. Again, let $M$ be a $n$-dimensional torus, $n>2$, and let $I\subset\{1,\ldots, n\}$ be a set of indices  with $\lvert I \rvert\ge 1$. We denote with $J=\{1,\ldots, n\}\setminus I$ and we consider the following equation:
\begin{equation}\label{eq:46}
\left(1+\sum_{i\in I}u_{ii}+G(\nabla u)\right)\left(1+\sum_{j\in J}u_{jj}+F(\nabla u)\right)- \sum_{\substack{i\in I \\ j\in J}}u_{ij}^2=e^{f}\,,
\end{equation}
 where, again, $f\in C^{\infty}(M)$ and $F$ and $G$ are smooth functions of the gradient of $u$ such that $F(0)=G(0)=0$. 

Clearly, equation (\ref{eq:46}) is the more general equation we can consider within this class. However, by a reorder of the coordinates, we can  study a simpler equation. Indeed, we assume that $\lvert I\rvert=k\ge 1$ and, using a diffeomorphism that reorders the coordinates, equation (\ref{eq:46}) is equivalent to the following:
\begin{equation}\label{eq:37}
\left(1+\sum_{i=n-k+1}^{n}u_{ii}+ G(\nabla u)\right)\left(1+\sum_{j=1}^{n-k}u_{jj}+ F(\nabla u)\right)-\sum_{i=n-k+1}^{n}\sum_{j=1}^{n-k}u_{ij}^2=e^{f}\,.
\end{equation} 
As we did in the previous sections, for the sake of simplicity,  from now on, we will denote
$$
A=1+\sum_{i=n-k+1}^{n}u_{ii}+ G(\nabla u)\,,\quad  B=1+\sum_{j=1}^{n-k}u_{jj}+ F(\nabla u)\,,
$$
 so that equation (\ref{eq:37}) can be rewritten as follows:
 \begin{equation}\label{eq:38}
 AB-\sum_{i=n-k+1}^{n}\sum_{j=1}^{n-k}u_{ij}^2=e^{f}\,.
 \end{equation}
 Obviously, we can assume, up to rename $A$ and $B$, that $k\le n-k.$

 %A first difference between (\ref{eq:3}) and (\ref{eq:38}) is the positivity of $A$ and $B$.
  Again, as in the previous case, $AB>0$ and, then,  $A$ and $B$ have the same sign on $M$. Then, we can consider $p\in M$ as the point where a solution $u\in C^{2}_0(M)$ attains its minimum and apply the Silvester Criterion for semidefinite matrixes. This yields to the fact that $u_{ii}(p)\ge 0 $, $\forall i=1,\ldots, n $. So, $A(p)>0$. Then, $A$ and $B$ are both positive on $M$.

The next step is to prove ellipticity for (\ref{eq:38}). As in Proposition \ref{prop:1.3},  we can define $\Phi^k\colon C^2(M)\to C^0(M)$ such that, $\forall v\in C^{2}(M)$, $$
\Phi^k(v)=A_vB_v-\sum_{i=n-k+1}^{n}\sum_{j=1}^{n-k}v_{ij}^2\,.
$$
 It is easy to prove that, if $u\in C_0^{2}(M)$ is a solution of (\ref{eq:38}), then $$
 d_u\Phi^k(v)=A_u\sum_{j=1}^{n-k}v_{jj}+B_u\sum_{i=n-k+1}^{n}v_{ii}-2\sum_{i=n-k+1}^{n}\sum_{j=1}^{n-k}u_{ij}v_{ij}+A_u\frac{d}{dt}F(\nabla u+tv)_{|_{t=0}}+B_u\frac{d}{dt}G(\nabla u+tv)_{|_{t=0}}\,,
 $$
  so, the matrix which represents the symbol of $d_u\Phi$ is $$
  P_u^k=\begin{pmatrix}A{\rm Id}_{n-k} & -C \\ -C^{t} & B{\rm Id}_{k}\end{pmatrix}\,,
  $$
   where ${\rm Id}_{n-k}$ and ${\rm Id}_{k}$ are, respectively, the identity in ${\rm GL}(n-k, \mathbb R)$ and ${\rm GL}(k, \mathbb R)$ and $C\in M(n-k,k,\mathbb{R})$  is such that 
   $$
   C_{st}=u_{s,n-k+t}\,,\quad  \forall s=1,\ldots, n-k\,, \quad \forall t=1,\ldots, k\,.
   $$
 In this case, for every vector $V=(V_{n-k}, V_k)\in \mathbb R^{n}$, we have that 
 $$
 P_u^kV=(A V_{n-k}-CV_k, - C^tV_{n-k}+B V_k)\,.
 $$ Then, denoted with  $\langle\cdot,\cdot\rangle_n$ the standard inner product on $\mathbb R^n$, we have
 $$
 \langle P_u^kV,  V\rangle_n=A\lvert V_{n-k}\rvert^2+B\lvert V_k\rvert^2-\langle CV_k,  V_{n-k} \rangle_{n-k}- \langle C^tV_{n-k}, V_k \rangle_{k}=A\lvert V_{n-k}\rvert^2+B\lvert V_{k}\rvert^2-2\langle CV_k,  V_{n-k}\rangle_{n-k}\,.
 $$  
 %On the other hand $CV_r=\sum_{j=1}^rC_{j \bar i  }(V_r)_j$ and $\langle CV_r, \bar V_s\rangle=\sum_{j=1}^s\sum_{i=1}^rC_{i \bar j  }(V_r)_i(V_s)_{\bar j }$. 
 But, we know that 
 $$
\langle CV_k,  V_{n-k}\rangle_{n-k}\le  \lvert\langle CV_k,  V_{n-k}\rangle_{n-k}\rvert\le \lvert CV_{k}\rvert\lvert V_{n-k}\rvert\le \lVert C\rVert\lvert V_k\rvert\lvert V_{n-k}\rvert\,,
 $$ where $\lVert C\rVert=\left( \sum_{s=1}^{n-k}\sum_{t=1}^k C_{st }^2\right)^{\frac 12}$. This implies that 
 $$
  \langle P_u^kV,  V\rangle_n\ge A\lvert V_{n-k}\rvert^2+B\lvert V_k\rvert^2-2\lVert C\rVert\lvert V_k\rvert\lvert V_{n-k}\rvert =(\sqrt{A}\lvert V_{n-k}\rvert- \sqrt{B}\lvert V_k\rvert)^2+2(\sqrt{AB}- \lVert C\rVert)\lvert V_k\rvert\lvert V_{n-k}\rvert
 $$ which is positive if $\sqrt{AB}- \lVert C\rVert>0$. On the order hand,  using \eqref{eq:38}, we have that 
 $$
 AB=e^{f}+ \sum_{i=n-k+1}^{n}\sum_{j=1}^{n-k}u_{ij}^2=e^f+\lVert C\rVert^2
 $$  which implies  
 $$
AB> \lVert C\rVert_2^2\,, 
 $$  giving us the claim. 
    Ellipticity  guarantees the fact that, $\forall \zeta\in \mathbb{R}^n$: 
   $$
   A\sum_{j=1}^{n-k}\zeta_j^2+B\sum_{i=n-k+1}^{n}\zeta_i^2-2\sum_{i=n-k+1}^n\sum_{j=1}^{n-k}u_{ij}\zeta_i\zeta_j\ge 0\,. 
   $$
   The same arguments as the ones used to prove, respectively, Theorem \ref{thm:1.1}, Theorem \ref{thm:2.4} and Theorem \ref{thm:3.5} yield  openness and  uniqueness, the $C^0$ estimate and the higher order estimates assuming an estimate similar to the one in Theorem \ref{thm:3.1}. So, it remains to prove an analogous of Theorem \ref{thm:3.1} for (\ref{eq:38}). As one may notice from the proof of Theorem \ref{thm:3.1}, in order to obtain the estimate we want, it is sufficient to prove an inequality similar to (\ref{eq:56}) and repeat the same argument. 
Then, to obtain that, $\forall j=1,\ldots, n-k $, we can choose $\zeta^{j}\in \mathbb{R}^n$  such that, $\forall j'=1,\ldots, n-k,\,\, j'\ne j $, $\zeta_{j'}^j=0$, $\zeta_j^j=\sqrt{\eta_{jj}}$ and, $\forall i=n-k+1,\ldots, n$, $\zeta_{i}^j=sign(u_{ij})\sqrt{\eta_{ii}}$ and apply the ellipticity with $\zeta^j$. We obtain that:
$$
A\eta_{jj}+B\sum_{i=n-k+1}^n\eta_{ii}-2\sum_{i=n-k+1}^n\lvert u_{ij}\rvert\sqrt{\eta_{ii}\eta_{jj}}\ge 0\,. 
$$
 Then, we can sum these $n-k$ inequalities and obtain 
 \begin{equation}\label{eq:61}
 A\sum_{j=1}^{n-k}\eta_{jj}+(n-k)B\sum_{i=n-k+1}^n\eta_{ii}-2\sum_{i=n-k+1}^n\sum_{j=1}^{n-k}\lvert u_{ij}\rvert\sqrt{\eta_{ii}\eta_{jj}}\ge 0\,,
 \end{equation} 
 which is extremely similar to (\ref{eq:56}) but, due to the presence of the factor $n-k$, it does not guarantee what we are looking for. So, it remains to understand if  (\ref{eq:61}) can be improved in order to obtain the $C^2$ estimate that is needed.

 \medskip
 The author has no conflicts of interest to declare that are relevant to the content of this article.

\end{document}